\documentclass[11pt]{amsart}
\usepackage{graphicx}
\usepackage[usenames]{color}
\usepackage{amssymb,amscd}\usepackage{amsfonts,mathrsfs}
\usepackage{graphicx,tikz}
\usetikzlibrary{arrows,automata,chains,fit,shapes}
\newcommand\conv{\otimes}
\newcommand{\ou}{\overline{u}}
\newcommand{\ov}{\overline{v}}

\newcommand\Linf{L_{\infty}}

\newtheorem{theorem}{Theorem}[section]

\newtheorem{lemma}[theorem]{Lemma}

\newtheorem{defn}[theorem]{Definition}

\newtheorem{prop}[theorem]{Proposition}

\title{Thompson's group $F$ is 1-counter graph automatic}
\author{Murray Elder and Jennifer Taback}
\address{School of Mathematical and Physical Sciences,
The University of Newcastle,
Callaghan NSW 2308, Australia}
\email{murrayelder@gmail.com}
\address{Department of Mathematics,
Bowdoin College, Brunswick, ME 04011} \email{jtaback@bowdoin.edu}
\thanks{The first author is supported by Australian Research Council grant  FT110100178.  The second author acknowledges support from National Science Foundation grant DMS-1105407 and Simons Foundation grant 31736 to Bowdoin College.  The authors wish to thank Sean Cleary, Bob Gilman, Alexei Miasnikov and especially Sharif Younes for many helpful discussions during the writing of this paper.}
\keywords{Thompson's group, automatic group, graph automatic group, $\mathcal C$-graph automatic group}
\subjclass[2010]{20F65, 68Q45}

\begin{document}

\maketitle

\begin{abstract}
It is not known whether Thompson's group $F$ is automatic. With the recent extensions of the notion of an automatic group to graph automatic by Kharlampovich, Khoussainov and Miasnikov and then to ${\mathcal C}$-graph automatic by the authors, a compelling question is whether $F$ is graph automatic or ${\mathcal C}$-graph automatic for an appropriate language class ${\mathcal C}$.  The extended definitions allow the use of a symbol alphabet for the normal form language, replacing the dependence on generating set.  In this paper we construct a $1$-counter graph automatic structure for $F$ based on the standard infinite normal form for group elements.
\end{abstract}

\section{Introduction}\label{sec:intro}
The notion of an automatic group was introduced in the 1990s based on ideas of Thurston, Cannon, Gilman, Epstein and Holt in the hopes of categorizing the fundamental groups of compact 3-manifolds. As not all nilpotent groups proved to be automatic, the definition required expansion to make it more robust.  A simultaneous motivation behind this definition was to use the automatic structure to ease computation within these groups.  To that end, it was shown  that automatic groups are finitely presented, have word problem solvable in quadratic time and possess at most a quadratic Dehn function \cite{WordProc}.  In an automatic group there is a natural set of quasigeodesic normal forms which define the structure, and this normal form representative for any group element can always be found in quadratic time.

It is not known whether Thompson's group $F$ is automatic.  Guba and Sapir present a regular normal form for elements of $F$ and prove that the Dehn function of $F$ is quadratic in \cite{GS}.  The word problem in $F$ is solvable in $O(n \log n )$ time (see, for example, \cite{SU}) and it is shown in \cite{BG} that $F$ is of type $FP_{\infty}$.   However, no one has been able to prove that the group is (or is not) automatic.  In \cite{CT} it is shown that $F$ cannot have a regular combing by geodesics with respect to the standard finite generating set, while in \cite{CHST} it is shown that $F$ admits a tame $1$-combing by quasigeodesics for this generating set, with a linear tameness function.

In \cite{KKM} Kharlampovich, Khoussainov and Miasnikov expand the definition of an automatic group to that of a graph automatic group by allowing the normal forms for elements to be expressed in terms of a symbol alphabet, rather than the group generators.  The precise definition is given below in Section \ref{sec:graph-aut}.  Groups which are not automatic, but are graph automatic with this definition include the solvable Baumslag-Solitar groups and finitely generated groups of nilpotency class  2. Graph automatic groups retain the properties that the word problem is solvable in quadratic time, and the normal form for any group element is computable in quadratic time.

The concept of a graph automatic group was extended by the authors in \cite{ET} to a ${\mathcal C}$-graph automatic group, where all languages comprising the structure are required to be in the language class ${\mathcal C}$.  This definition is stated precisely in Section \ref{sec:graph-aut}.

The introduction of a symbol alphabet makes $F$ a natural candidate for the class of graph (or ${\mathcal C}$-graph) automatic groups, as there are two natural symbol alphabets associated to this group.

\medskip

\begin{enumerate}
\item The first symbol alphabet is based on the standard infinite normal form for elements of $F$, and in this paper we prove that there is a language of quasi-geodesic normal forms over this alphabet which forms a $1$-counter graph automatic structure for $F$.  Moreover, $F$ is not graph automatic with respect to this language of normal forms.

    \medskip

\item The second symbol alphabet is based on the combinatorial ``caret types" of nodes in the reduced pair of trees representing a given element.  In \cite{TY}, Younes and the second author define a language of quasigeodesic normal forms for elements of $F$ based on this alphabet which they show to be the basis of a $3$-counter graph automatic structure for $F$, although they note that the number of counters can be reduced to $2$.  They also show that $F$ is not graph automatic with respect to this language of normal forms.
\end{enumerate}
Neither of these results rules out the possibility that $F$ is graph automatic with respect to a different normal form, or alternate symbol alphabet.

We remark that a third possibility for constructing a $\mathcal C$-graph automatic structure for $F$ would be to use the regular normal form language of   Guba and Sapir \cite{GS}.  In preparing this article we considered the complexity of the multiplier automata for this structure and found that the multiplier automaton for the generator $x_1$  would require at least $4$ counters, and hence did not include the details in this article.

This paper is organized as follows.  In Section \ref{sec:background}  we recall the definition of a ${\mathcal C}$-graph automatic group introduced in \cite{ET} which generalizes the definition of a graph automatic group introduced by Kharlampovich, Khoussainov and Miasnikov in \cite{KKM}, and present a brief introduction to Thompson's group $F$.
Section \ref{sec:Infnf} is devoted to describing a language of normal forms based on the standard infinite normal form for elements of $F$, and the resulting 1-counter graph automatic structure.

\section{Background}\label{sec:background}

We present some background material on ${\mathcal C}$-graph automatic groups as well as a brief introduction to Thompson's group $F$.

\subsection{Generalizations of Automaticity}\label{sec:graph-aut}

The following definitions are taken from \cite{ET} and include the definition of a graph automatic group given in \cite{KKM}.  We begin with the definition of a convolution of strings, following \cite{KKM}.

Let $G$ be a group with symmetric  generating set $X$, and $\Lambda$ a finite set of symbols. In general we do not assume that $X$ is finite.
The number of symbols (letters) in a word $u\in\Lambda^*$ is denoted $|u|_{\Lambda}$.

\begin{defn}[convolution;  Definition 2.3 of \cite{KKM}]
Let $\Lambda$ be a finite set of symbols, $\diamond$  a symbol not in $\Lambda$,  and let $L_1,\dots, L_k$ be a finite set of languages over $\Lambda$. Set  $\Lambda_{\diamond}=\Lambda\cup\{\diamond\}$. Define the {\em convolution of a tuple} $(w_1,\dots, w_k)\in L_1\times \dots \times L_k$ to be the string $\otimes(w_1,\dots, w_k)$ of length $\max |w_i|_{\Lambda}$ over the alphabet $\left(\Lambda_{\diamond}\right)^k$  as follows.
The $i$th symbol of the string is
\[\left(\begin{array}{c}
\lambda_1\\
\vdots\\
\lambda_k
\end{array}\right)\]
where $\lambda_j$ is the $i$th letter of $w_j$ if $i\leq |w_j|_{\Lambda}$ and $\diamond$ otherwise.
Then  \[\otimes(L_1,\dots, L_k)=\left\{\otimes(w_1,\dots, w_k) \mid w_i\in L_i\right\}.\]
\end{defn}
We note that the convolution of regular languages is again a regular language.

As an example, if $w_1=aa, w_2=bbb$ and $w_3=a$ then
\[\otimes(w_1,w_2,w_3)=\left(\begin{array}{c}
a\\
b\\
a
\end{array}\right)
\left(\begin{array}{c}
a\\
b\\
\diamond
\end{array}\right)
\left(\begin{array}{c}
\diamond\\
b\\
\diamond
\end{array}\right)\]

When $L_i=\Lambda^*$ for all $i$ the definition in \cite{KKM} is recovered.

\begin{defn}[quasigeodesic normal form]
A  {\em normal form for $(G,X,\Lambda)$} is a set of words $L\subseteq \Lambda^*$ in bijection with $G$. A normal form $L$ is
  {\em quasigeodesic}  if there is a constant $D$ so  that $$|u|_{\Lambda}\leq D(||u||_X+1)$$ for each $u\in L$,  where $||u||_X$ is the length of a geodesic  in $X^*$ for the group element represented by  $u$.
\end{defn}
The $||u||_X+1$ in the definition allows for normal forms where the identity of the group is represented by a nonempty string of length at most $D$.  We denote the image of $u\in L$ under the bijection with $G$ by $\ou$.

\begin{defn}[$\mathcal C$-graph automatic group] \label{def:Cgraph}
Let $\mathcal C$ be a formal language class, $(G,X)$ a group and symmetric  generating set, and $\Lambda$ a finite set of symbols.  We say that $(G,X,\Lambda)$ is $\mathcal C${\em -graph automatic} if there is a  normal form $L \subset \Lambda^*$ in the language class $\mathcal C$, such that for each $x\in X$ the language $L_x=\{\otimes(u,v) \mid u,v\in L,  \ov =_G \ou x\}$ is in the class $\mathcal C$.
\end{defn}

If we take ${\mathcal C}$=$\{$regular languages$\}$ we recover the definition of a graph automatic group given in \cite{KKM}.  If we further require $\Lambda$ to be the set $X$ of group generators, we obtain the original definition of an automatic group.

In this paper we will consider the case when ${\mathcal C}$ denotes the set of non-blind deterministic 1-counter languages, which are defined as follows.

\begin{defn}[1-counter automaton]
A {\em non-blind deterministic $1$-counter automaton} is
  a deterministic finite
state automaton augmented with an integer counter: the counter is
initialized to zero, and can be incremented, decremented, compared to zero and set to zero during operation.  For each configuration of the machine and subsequent input letter, there is at most one possible move. The automaton accepts a word exactly if upon reading the word it  reaches
an accepting state with the counter equal to zero.
\end{defn}

In drawing a 1-counter automaton, we label transitions by the input letter to be read, with subscripts to denote the possible counter instructions:
\begin{itemize}
\item $=0$ to indicate the edge may only be traversed if the value of the counter is $0$.
\item $>0$ (resp. $<$) to indicate the edge may only be traversed if the value of the counter is greater (resp. less) than $0$.
\item $+1$ to increment the counter by $1$.
\item $-1$ to decrement the counter by $1$.
\end{itemize}

The class $\mathscr C_1$ of 1-counter languages is closed  under homomorphism, inverse homomorphism, intersection with regular languages, and finite intersection, and is strictly contained in the class of context-free languages (see \cite{HU}, for example).

We end this section with a lemma which is used to streamline several proofs in later sections, and is a restatement of Lemma 2.6 of \cite{TY}.

\begin{lemma}[\cite{TY}, Lemma 2.6] \label{L:RegularLanguages}
Fix a symbol alphabet $\Lambda$ and let $L_1$ and $L_2$ be regular languages over $\Lambda$, and $x \in \Lambda^*$ a fixed word.  Then the set
$$\{\conv(zw,zxw) \mid z \in L_1, \ x \in \Lambda^*, \ w \in L_2 \}$$ is a regular language.
\end{lemma}

If the language of normal forms for a $k$-counter-graph automatic group is additionally quasigeodesic, it is proven in \cite{ET} that given a string of group generators of length $n$, the normal form word can be computed in time $O(n^{2k+2})$.  It is proven in \cite{WordProc} that this is true for any automatic group, and in \cite{KKM} that any graph-automatic group enjoys this property.  In \cite{ET} an example is given to show that a ${\mathcal C}$-graph automatic structure need not have a quasigeodesic normal form language.

\subsection{Thompson's group $F$}\label{sec:F}

Thompson's group $F$ can be equivalently viewed from three perspectives:
\begin{enumerate}
\item As the group defined by the infinite presentation
$$ {\mathcal P}_{inf}=\langle x_0,x_1, x_2\cdots | x_jx_i = x_ix_{j+1} \text{ whenever } i<j \rangle $$
or the finite presentation
$$ {\mathcal P}_{fin}=\langle x_0,x_1| [x_0^{-1}x_1,x_0^{-1}x_1x_0],[x_0^{-1}x_1,x_0^{-2}x_1x_0^2] \rangle. $$

\medskip

\item As the set of piecewise linear homeomorphisms of the interval $[0,1]$ satisfying
\begin{enumerate}
\item each homeomorphism has finitely many linear pieces,
\item all breakpoints have coordinates which are dyadic rationals, and
\item all slopes are powers of two.
\end{enumerate}

\medskip

\item As the set of pairs of reduced finite rooted binary trees with the same number of nodes, or carets.
\end{enumerate}
For the equivalence of these three interpretations of this group, as well as a more robust introduction to the group, we refer the reader to \cite{CFP}.

There is a standard normal form for elements of $F$ with respect to the infinite generating set, as discussed by Brown and Geoghegan \cite{BG}.  Any element $g \in F$ can be written uniquely, by applying the relations from the presentation ${\mathcal P}_{inf}$, as
\begin{equation}\label{eqn:inf-nf}
x_{i_0}^{e_0}x_{i_1}^{e_1} \cdots x_{i_m}^{e_m}x_{j_n}^{-f_n} \cdots x_{j_1}^{-f_1}x_{j_0}^{-f_0}
\end{equation}
where
\begin{enumerate}
\item $0 \leq i_0 < i_1 < i_2 < \cdots < i_m$ and $0 \leq j_0 < j_1 < j_2 < \cdots < j_n$,
\item $e_i,f_j > 0$ for all $i,j$, and
\item if $x_i$ and $x_i^{-1}$ are both present in the expression, then so is $x_{i+1}$ or $x_{i+1}^{-1}$.
\end{enumerate}
We will refer to this as the standard infinite normal form for $g \in F$; an expression in this form is called {\em reduced}.  If an expression of the form in \eqref{eqn:inf-nf} contains $x_i$ and $x_i^{-1}$ but not $x_{i+1}$ or $x_{i+1}^{-1}$ for some $i$, it is called {\em unreduced}.  Relations from ${\mathcal P}_{inf}$ must be applied to obtain an equivalent expression satisfying the third condition above, and then the resulting expression is reduced.

\section{A 1-counter language of normal forms for elements of $F$}\label{sec:Infnf}

In this section we define a language of normal forms, based on the standard infinite normal form for elements of $F$, which yields a $1$-counter graph automatic structure for $F$.  If $g \in F$, let $w$ denote the infinite normal form for $g$ given in \eqref{eqn:inf-nf}.  From the expression $w$, define a normal form over the finite alphabet $\{\#,a,b\}$ as follows.
\begin{enumerate}
\item First, we require that every generator between $x_0$ and $x_M$ appear twice
     in the expression in \eqref{eqn:inf-nf}, where $ M=\max\{i_m,j_n\}$.  To accomplish this, insert $x_t^0$ for any index $t \leq M$ not appearing in \eqref{eqn:inf-nf}. This yields a word of the form
\begin{equation}\label{eqn:inf-nf2}
 x_{0}^{r_0}x_{1}^{r_1}x_{2}^{r_2} \cdots x_{M}^{r_M}x_{M}^{-s_M} \cdots x_{2}^{-s_2}x_{1}^{-s_1}x_{0}^{-s_0}
 \end{equation}
where $r_i,s_i\geq 0$, exactly one of $r_M,s_M$ is nonzero, and $r_is_i>0$ implies $r_{i+1}+s_{i+1}>0$.

\item Second, rewrite the expression in \eqref{eqn:inf-nf2} in the form $$a^{r_0}b^{s_0}\#a^{r_1}b^{s_1}\#\dots \#a^{r_M}b^{s_M}$$
where $r_i,s_i\geq 0$, exactly one of $r_M,s_M$ is nonzero, and $r_is_j>0$ implies $r_{i+1}+s_{i+1}>0$.
\end{enumerate}

Define $\Linf$ to be the set of all such strings satisfying these conditions on $r_i$ and $s_i$, together with the empty string.

As an example, the element with standard infinite normal form $$x_1^2x_4^3x_8^{-1}x_5^{-6}x_4^{-2}$$
has the following representative in the language $\Linf$:
$$\#aa   \# \#  \#  aaabb \#   bbbbbb \#   \#   \#      b.$$

Note that if the largest index in the standard infinite normal form  string $w$ is $k$, then the normal form string in $\Linf$ representing $w$ contains exactly $k$ $\#$ symbols.

\begin{figure}[ht!]
\begin{center}
\tikzset{every state/.style={minimum size=2em}} 
\begin{tikzpicture}[->,>=stealth',shorten >=1pt,auto, node distance=3cm,
semithick]

\node[state, accepting, initial] (S) {$q_0$};
\node[state] (A) [right of=S] {};
\node[state] (B) [right of=A] {};
\node[state, accepting] (C) [above of=A] {$q_1$};
\node[state, accepting] (D) [below of=A] {$q_2$};
\node[state] (E) [above of=B] {};

\path (S)
edge node {$\#$} (A)
edge [bend left] node {$a$} (C)
edge [bend right] node {$b$} (D);

\path (A) edge [loop right] node {$\#$} (A)
edge [bend left] node {$a$} (C)
edge [bend left] node [left]{$b$} (D);

\path (B) edge [bend left]node {$b$} (D)
edge node {$a$} (C);

\path (C) edge [loop above] node {$a$} (C)
edge [bend left] node[left] {$\#$} (A)
edge node {$b$} (E);

\path (D) edge [loop below] node {$b$} (D)
edge [bend left] node[left] {$\#$} (A);

\path (E) edge [loop above] node {$b$} (E)
edge  node {$\#$} (B);

\end{tikzpicture}
\caption{A finite state automaton accepting the language $\Linf$. Accept states are $q_0,q_1,q_2$.}
\label{fig:fsaLinfty}
\end{center}
\end{figure}
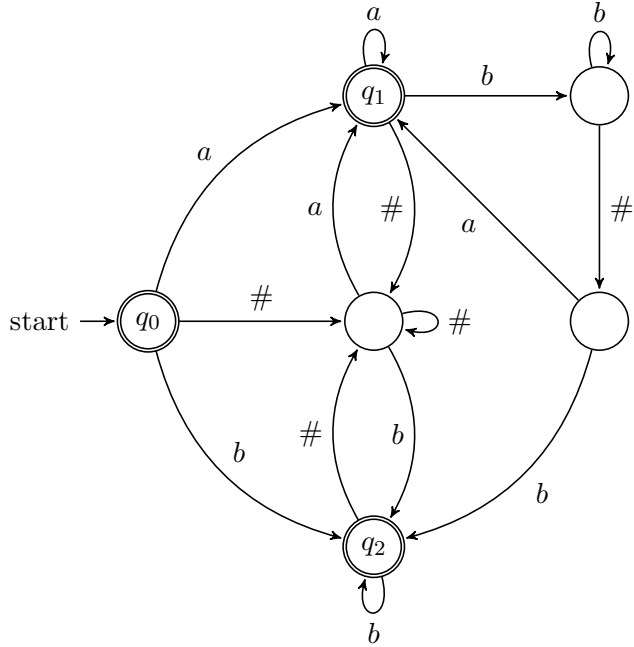

\begin{lemma}\label{lemma:regular-nf-inf}
The language $\Linf$ is a normal form for $F$, and is a regular language.
\end{lemma}
\begin{proof}
The procedure described above clearly gives a bijection between the set of standard infinite normal forms for elements of $F$ and the language $\Linf$.  To see that this language is regular, we give a finite state machine which accepts it in Figure~\ref{fig:fsaLinfty}. The automaton only accepts strings of the form $\varepsilon,a^i,b^i,w\#a^i,w\#b^i$ with $i>0$ and $w$ a word in $\{a,b,\#\}^*$  with no subword $ba$, and if $w$ has a subword $ab$ then any maximal substring of the form
$$a^{k_1}b^{k_2} \#a^{k_2}b^{l_2} \# \cdots \# a^{k_z}b^{l_z} \# $$
where all $k_i,l_i >0$ is immediately followed by $a^j\#$ or $b^j\#$.
\end{proof}

To prove that this language of normal forms is quasigeodesic we quote a proposition proven by Burillo in \cite{Burillo}.
\begin{prop}[\cite{Burillo}, Prop. 2] \label{lemma:Burillo}
Let $g \in F$ have standard infinite normal form
$$x_{i_1}^{e_1}x_{i_2}^{e_2} \cdots x_{i_m}^{e_m}x_{j_n}^{-f_n} \cdots x_{j_2}^{-f_2}x_{j_1}^{-f_1}$$
and define
$$D= \sum_{l=1}^m e_l + \sum_{l=1}^n f_l + i_m+j_n.$$
Then $\frac{D}{6}-2 \leq l(g) \leq 3D$ where $l(g)$ denotes the word length of $g$ with respect to the finite generating set $\{x_0,x_1\}$.
\end{prop}

Using Proposition \ref{lemma:Burillo} we prove the following proposition.

\begin{prop}\label{lemma:quasi-geo}
 The language  $\Linf$ is quasigeodesic.
\end{prop}
\begin{proof}
Let $g$ have standard infinite normal form $x_{i_1}^{e_1}x_{i_2}^{e_2} \cdots x_{i_m}^{e_m}x_{j_n}^{-f_n} \cdots x_{j_2}^{-f_2}x_{j_1}^{-f_1}$ which we alter as above to the expression in \eqref{eqn:inf-nf2}:
$$x_{0}^{r_0}x_{1}^{r_1} \cdots x_{M}^{r_M}x_{M}^{-s_M} \cdots x_{1}^{-s_1}x_{0}^{-s_0}$$
in which each index $0$ through $M= \max\{i_m,j_n\}$ appears twice, although possibly with zero exponent.  If $r_k \neq 0$ (resp. $s_k \neq 0$) then $r_k=e_i$ (resp. $s_k=f_i$), for some $i \leq M$.
The length of the resulting normal form for $g$ in the language $\Linf$ is then
$$\sum_{l=0}^M (r_l + s_l) + M =\sum_{l=1}^m e_l + \sum_{l=1}^n f_l +M = D'. $$
It is clear that $D \leq D' \leq 2D$, as $M= \max\{i_m,j_n\}$. Combining this with the inequality in Proposition \ref{lemma:Burillo} we obtain
$$\frac{D'}{12}-2 \leq \frac{D}{6}-2 \leq l(g) \leq 3D \leq 3D'$$
and hence $\Linf$ is a language of quasigeodesic normal forms for elements of $F$.
\end{proof}

Next we consider the multiplier languages  $L_{x_0^{-1}}$ and $L_{x_1^{-1}}$.    It is proven in \cite{ET} that for $\mathscr R=\{$regular languages$\}$ or $\mathscr C_1=\{$1-counter languages$\}$ we have $L_{x}\in \mathscr R$  (resp. $\mathscr C_1$) if and only if $L_{x^{-1}}\in \mathscr R$ (resp. $\mathscr C_1$).  Hence it suffices to consider only the multiplier languages  $L_{x_0^{-1}}$ and $L_{x_1^{-1}}$.

\begin{prop}\label{prop:xzero}
The language $L_{x_0^{-1}}=\{\otimes(u,v) \mid u,v\in \Linf, \overline ux_0^{-1}=_F\overline v\}$ is regular.
\end{prop}
\begin{proof}
Let $g \in F$ with $\Linf$ normal form $$u=a^{r_0}b^{s_0}\#a^{r_1}b^{s_1}\#\dots \#a^{r_M}b^{s_M}$$ and infinite normal form $w$. We consider two cases, depending on whether $wx_0^{-1}$ is already in infinite normal form, or must be simplified in order to obtain the infinite normal form.

First suppose that $wx_0^{-1}$ is in infinite normal form.  This occurs if and only if any one of the following conditions holds:
\begin{enumerate}
\item The expression $w$ contains no $x_0$ terms to a positive exponent, that is, $r_0=0$.
\item The expression $w$ contains $x_1$ to a nonzero power, that is, $r_1 \neq 0$ or $s_1 \neq 0$.
\item The expression $w$ contains $x_0$ to a negative power, that is, $s_0 \neq 0$.
\end{enumerate}
Write $u =a^{r_0} b^{s_0} \gamma  \in \Linf$ where $\gamma$ is empty or has initial letter $\#$. Then $v=a^{r_0} b^{s_0+1}\gamma$ is the word in $L_\infty$ corresponding to  $ux_0^{-1}$.

Next suppose that $wx_0$ is not in infinite normal form, so that none of the above conditions hold.  That is, $u=a^{r_0}\#\#a^{r_2}b^{s_2}\gamma$ where $r_0>0$ and $\gamma$ is either empty (and one of $r_2,s_2$ is zero) or begins with $\#$.
In this case $wx_0^{-1}$ is an unreduced expression and has the form $x_0^{e_0} x_i^{e_i} \cdots x_j^{-f_j} x_0^{-1}$ for some $i,j > 1$.  Rewrite this expression as $wx_0^{-1} = x_0^{e_0}  \beta x_0^{-1}$ where $\beta = x_i^{e_i} \cdots x_j^{-f_j}$.

From the infinite presentation for $F$ we have the relations $x_0x_{j+1}x_0^{-1} = x_j$ for $j \geq 1$.  We apply this repeatedly to the word $x_0^{e_0}  \beta x_0^{-1}$ to obtain $x_0^{e_0-1}  \beta' $ where $x_i$ in $\beta$ is replaced by $x_{i-1}$ to obtain $\beta'$.  As none of the exponents are altered in this process, as words in $\Linf$ we have
\begin{itemize}
\item $u = a^{r_0} \# \#  a^{r_2} b^{s_2} \gamma$, and
\item $v = a^{r_0} \#   a^{r_2} b^{s_2} \gamma$
\end{itemize}
where $v$ denotes the string in $\Linf$ corresponding to $wx_0^{-1}$.

Under either assumption, let $K$ be the set of convolutions $\otimes(u,v)$ where $u$ and $v$ have either form listed above, altered so that any string $\gamma$ is allowed to lie in the set $\{a,b,\#\}^*$.  It follows from Lemma~\ref{L:RegularLanguages} that $K$ is a regular language.  Then $L_{x_0^{-1}} = K \cap \otimes(\Linf,\Linf)$ is a regular language and accepts exactly those convolutions $\otimes(u,v)$ where $u,v \in \Linf$ and $\overline ux_0^{-1}=_F\overline v$.
\end{proof}

We now show that the multiplier language $L_{x_1^{-1}}$ is  a deterministic non-blind $1$-counter language. In Lemma~\ref{lem:notRegLinfty} we prove this language cannot be regular.

\begin{prop}\label{prop:xone}
The language $L_{x_1^{-1}}=\{\otimes(u,v) \mid u,v\in \Linf, \overline ux_1^{-1}=_F\overline v\}$ is a deterministic non-blind $1$-counter language.
\end{prop}
\begin{proof}
Let $g \in F$ have $\Linf$ normal form $$u=a^{r_0}b^{s_0}\#a^{r_1}b^{s_1}\#\dots \#a^{r_M}b^{s_M}$$ and infinite normal form $w$.  We will always use $v$ to denote the representative in $\Linf$ of $gx_1^{-1}$.
We  consider several cases depending on the value of certain exponents.

\bigskip

\noindent Case 1: Suppose first that $s_0 = 0$, that is, the infinite normal form $w$ for $g$ does not contain $x_0$ to a negative exponent. We consider three subcases.

\medskip

\noindent
Case 1.1:  If $u=a^{r_0}$ for $r_0\geq 0$, then $v=a^{r_0}\#b$.

\medskip

\noindent
Case 1.2: If $u$ contains at least one $\#$ symbol, then $wx_1^{-1}$ is the infinite normal form for $gx_1^{-1}$ if at least one of the following conditions holds:
\begin{enumerate}
\item The expression $w$ contains no $x_1$ terms to a positive exponent, that is, $r_1=0$.
\item The expression $w$ contains an $x_1$ to a negative power, that is, $s_1 \neq 0$.
\item The expression $w$ contains $x_2$ to a nonzero power, that is, $r_2 \neq 0$ or $s_2 \neq 0$.
\end{enumerate}
If we write $u = a^{r_0} \# a^{r_1} b^{s_1} \gamma \in \Linf$, where $\gamma$  is empty or begins with $\#$,
then $v = a^{r_0} \#  a^{r_1} b^{s_1+1} \gamma$.

\medskip

\noindent
Case 1.3:  If none of the previous conditions hold  we must have $u=a^{r_0}\#a^{r_1}\gamma \in \Linf$ with $r_1>0$ and $\gamma$ either empty or $\gamma = \#\# \gamma'$. The corresponding infinite normal form is then $w= x_0^{r_0} x_1^{r_1} \eta$ where $\eta$ is either empty or  $\eta=x_i^{r_i} \cdots x_j^{-s_j}$ for some $i,j > 2$ and $r_k,s_l>0$.  Then $wx_1^{-1}= x_0^{r_0} x_1^{r_1} \eta x_1^{-1}$.  As in Proposition \ref{prop:xzero}, this simplifies to $wx_1^{-1}=x_0^{r_0}x_1^{r_1-1} \eta'$, where $x_i^{\pm 1}$ in $\eta$ is replaced by $x_{i-1}^{\pm 1}$ to obtain $\eta'$. Since $\eta$ and $\eta'$ have the same sequence of exponents, we have, for $r_1>0$
\begin{itemize}
		\item $u = a^{r_0} \#  a^{r_1} \gamma$, and
		\item \begin{enumerate}
\item $v =  a^{r_0} \#  a^{r_1-1} $ if $r_1>1$ and $\gamma$ is empty,  \item $v=a^{r_0}$ if $r_1=1$ and $\gamma$ is empty \item $v =  a^{r_0} \#  a^{r_1-1} \# \gamma'$ if $r_1>1$ and $\gamma=\#\# \gamma'$.  \end{enumerate}
	\end{itemize}

Let $K$ denote the regular language of convolutions $\otimes(u,v)$ arising from Case 1, with the string $\gamma$ replaced by any string in $\{a,b,\#\}^*$.  Then $K$ is the union of four languages, splitting Case 1.3 according to whether $\gamma$ is empty or not.  The languages of pairs $\otimes(u,v)$ arising from Cases 1.1 or 1.3 with $\gamma$ empty are clearly regular, and it follows from Lemma \ref{L:RegularLanguages} that the languages arising from Cases 1.2 and 1.3, with $\gamma$ nonempty in the latter, are also regular.  Thus $K$ is a regular language, and $L_{s_0=0}= K \cap \otimes(\Linf,\Linf)$ is a regular language, and contains exactly those convolutions of strings covered by Case 1.

\medskip

\noindent
Case 2: Next suppose that $s_0 \neq 0$, so $w$ ends in $x_0^{-1}$, and hence $wx_1^{-1}$ is not the infinite normal form for $gx_1^{-1}$.  We will describe $L_{s_0 \neq0}$, which is the set of all strings $\otimes(u,v) \in L_{x_1^{-1}}$ in which $u$ satisfies $s_0 \neq 0$.

We will apply the relation  $x_i^{-1}x_j^{-1} = x_{j+1}^{-1}x_i^{-1}$  for $i<j$ repeatedly, to ``push" the final $x_1^{-1}$ to the left in this expression, at the ``cost" of increasing its index.

Let
\begin{equation}\label{eqn:infnf}
w=x_{i_0}^{e_0}x_{i_1}^{e_1} \cdots x_{i_m}^{e_m}x_{j_n}^{-f_n} \cdots x_{j_1}^{-f_1}x_{j_0}^{-f_0}
\end{equation}
be the infinite normal form for $g$, where $f_0=s_0 \neq0$.
Applying the relation above $f_0$ times to the expression $wx_1^{-1}$ yields
$$x_{i_0}^{e_0}x_{i_1}^{e_1} \cdots x_{i_m}^{e_m}x_{j_n}^{-f_n} \cdots x_{j_2}^{-f_2}x_{j_1}^{-f_1}\left(x_{1+f_0}\right)^{-1}x_0^{-f_0}.$$
If $1+f_0 \leq j_1$ then this process is completed, and we must determine if the resulting word is reduced, or can be simplified further.  If $1+f_0 > j_1$ then we can apply the relation again $f_1$ times to obtain
$$ x_{i_0}^{e_0}x_{i_1}^{e_1} \cdots x_{i_m}^{e_m}x_{j_n}^{-f_n} \cdots x_{j_2}^{-f_2}\left(x_{1+f_0+f_1}\right)^{-1}x_{j_1}^{-f_1}x_0^{-f_0}.$$
We continue applying this relation until the first time we obtain $x_R^{-1}$, where either
\begin{itemize}
\item $R=1+f_0+f_1+\dots +f_n>j_n$, or
\item $R=1+f_0+f_1+\dots +f_t \leq j_{t+1}$ for some $1 \leq t \leq n-1$.
\end{itemize}

\bigskip

\noindent
Case 2.1: Suppose that $R>j_n$.  We must consider the relative values of $R$ and $M= \max\{i_m,j_m\}$.

First assume that $R>M$.  Then the infinite normal form expression for $gx_1^{-1}$ is
$$x_{i_0}^{e_0}x_{i_1}^{e_1} \cdots x_{i_m}^{e_m}x_R^{-1} x_{j_n}^{-f_n} \cdots x_{j_2}^{-f_2}x_{j_1}^{-f_1}x_0^{-f_0}$$
with $f_0 \neq 0$.  It follows that the representatives of the elements $g$ and $gx_1^{-1}$ in $\Linf$ are:
\begin{itemize}
\item $u=a^{r_0}b^{s_0}\gamma$ where $\gamma$ is either empty or starts with $\#$, ends with $a$ or $b$ and contains exactly $M$ $\#$ symbols.
\item $v=a^{r_0}b^{s_0} \gamma \#^{R-M}b$. Note that the number of $\#$ symbols in this word is $R$, as required.
\end{itemize}

The deterministic non-blind $1$-counter automaton given in Figure \ref{fig:L4} accepts exactly those strings of the form $\otimes(a^{r_0}b^{s_0} \gamma, a^{r_0}b^{s_0} \gamma \#^{R-M} b)$ in which the the suffix of the convolution has the correct number of ${ \# \choose \# }$ symbols and $\gamma \in \{a,b,\#\}^*$.  We denote this language by $L'$.  The automaton operates as follows.
\begin{itemize}\item
After reading ${a\choose a}^i {b\choose b}$ the value of the counter is set to $2$, and the automaton is in state $q_1$.
\item  For each ${b\choose b}$ read the counter is increased by 1, and for each ${\#\choose\#}$ the counter is decremented by 1, so that after reading a prefix containing $p$ copies of the symbol ${\#\choose\#}$, the value of the counter is equal to \[
1+f_0+f_1+\dots + f_{p-1}-p.\]
\item If the counter returns to zero while at state $q_1$ after having read $p$ copies of the symbol ${\#\choose\#}$, we are not in Case 2.1 and the input is rejected, since the edge leaving $q_1$ for state $q_2$ checks that the counter is positive.
\item Once the automaton has read the string $\otimes(a^{r_0}b^{s_0} \gamma, a^{r_0}b^{s_0} \gamma)$, the value of the counter is $1+f_0+f_1+\dots + f_n-M=R-M$ since we have read all the ${b\choose b}$ letters in $\otimes(a^{r_0}b^{s_0} \gamma, a^{r_0}b^{s_0} \gamma)$ and $M$ ${\#\choose \#}$ letters.  Recall that $M=\max\{i_m,j_n\}$.
From here
the input is accepted precisely if the remaining letters to be read are $\otimes(\epsilon, \#^{R-M}b)$, as verified by the automaton.
\end{itemize}
Then $L_{R>M} = L' \cap \otimes(\Linf,\Linf)$ is a deterministic non-blind $1$-counter automaton which accepts exactly those strings from Case 2.1 which satisfy $R>M$.

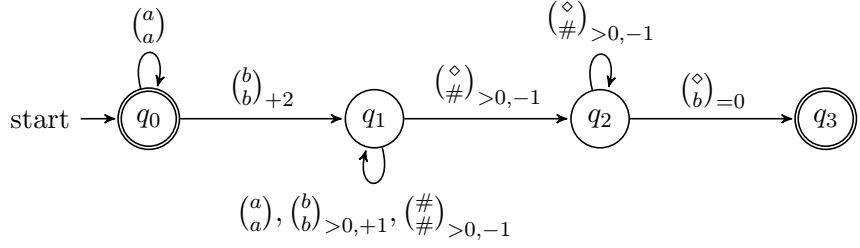
\begin{figure}[ht!]
\begin{center}
\tikzset{every state/.style={minimum size=2em}} 
\begin{tikzpicture}[->,>=stealth',shorten >=1pt,auto, node distance=3cm,
semithick]

\node[state, accepting, initial] (S) {$q_0$};
\node[state] (A) [right of=S] {$q_1$};
\node[state] (B) [right of=A] {$q_2$};
\node[state, accepting] (C) [right of=B] {$q_3$};

\path (S)
edge node {${b\choose b}_{+2}$} (A)
edge [loop above] node {${a\choose a}$} (S);

\path (A) edge [loop below] node {${a\choose a}, {b\choose b}_{>0,+1}, {\#\choose \#}_{>0, -1}$} (A)
edge node {${\diamond\choose \#}_{>0, -1}$} (B);

\path (B) edge node {${\diamond\choose b}_{=0}$} (C)
edge [loop above] node {${\diamond\choose \#}_{>0, -1}$} (B);

\end{tikzpicture}
\caption{A deterministic non-blind $1$-counter automaton used in Case 2.1 of the proof of Proposition \ref{prop:xone} when $R>M$.   The start state is $q_0$ and the accept state is $q_3$.}

\label{fig:L4}
\end{center}
\end{figure}

If $R=M$, which can occur if and only if $R=i_m$, we claim that the infinite normal form for $gx_1^{-1}$ is either
\begin{equation}\label{eq:em>1}
x_{i_0}^{e_0}x_{i_1}^{e_1} \cdots x_{i_{m}}^{e_m-1}x_{j_n}^{-f_n} \cdots x_{j_2}^{-f_2}x_{j_1}^{-f_1}x_0^{-f_0} \text{ if }e_m>1, \text{ or }
\end{equation}
\begin{equation}\label{eq:em=1}x_{i_0}^{e_0}x_{i_1}^{e_1} \cdots x_{i_{m-1}}^{e_{m-1}}x_{j_n}^{-f_n} \cdots x_2^{-f_2}x_{j_1}^{-f_1}x_0^{-f_0} \text{ if } e_m =1,
\end{equation}
that is, in \eqref{eq:em=1} we have $i_{m-1} \neq j_{n}$ and hence there is no additional cancelation of terms through application of relations from ${\mathcal P}_{inf}$.  We justify this statement as follows.

If $i_{m-1} = j_{n}$, as we began with a reduced expression in the infinite normal form for $g$, we must have $i_m = i_{m-1}+1 = j_n+1$.   In Case 2.1, it is always true that $$1+f_0 + \cdots + f_b \geq j_{b+1}+1$$ for all $0 \leq b \leq n$.  So $$1 + f_0 + f_1 + \cdots f_{n-1} \geq j_n +1$$ and it follows that $$R=1+f_0+f_1+\cdots f_n \geq j_n+2=i_m+1,$$ in which case $R \neq i_m$.  Thus  $i_{m-1} \neq j_{n}$ and \eqref{eq:em=1} is the infinite normal form for $gx_1^{-1}$.

Again letting $\overline{u}x_1^{-1} =_F \overline{v}$, in this case we have
\begin{itemize}
\item $u=\gamma \#^s a^{e_m}$, where $\gamma$ ends in $a$ or $b$ and
\item $v=\gamma \#^s a^{e_m-1}$ when $e_m>1$ and $v=\gamma \ $ otherwise.
\end{itemize}

The $1$-counter automaton in Figure \ref{fig:R=M} accepts exactly those convolutions of elements of the above forms with $\gamma \in \{a,b,\#\}^*$.  The top line of states and transitions is followed when $e_m>1$ and the bottom line when $e_m=1$.  Note that the value of the counter must equal $0$ when the difference in the words is detected by the automaton.  Intersecting the language accepted by this machine with $\otimes(\Linf,\Linf)$ yields a language $L_{R=M}$ consisting exactly of those convolutions of strings accepted in Case 2.1 with $R=M$.

\begin{figure}[ht!]
\begin{center}
\scalebox{.8}{
\tikzset{every state/.style={minimum size=2em}} 
\begin{tikzpicture}[->,>=stealth',shorten >=1pt,auto, node distance=4cm,
semithick]

\node[state, initial] (S) {$q_0$};
\node[state] (A) [right of=S] {$q_1$};
\node[state] (B) [right of=A] {$q_2$};
\node[state] (C) [right of=B] {$q_3$};
\node[state, accepting] (D) [below of=C] {$q_4$};
\node[state] (E) [below of = A] {$q_5$};
\node[state, accepting] (F) [right of = E] {$q_6$};

\path (S)
edge [loop above] node {${a\choose a}$} (S)
edge node {${b \choose b}_{+2}$} (A);

\path (A)
edge node {${\#\choose \#}_{>0,-1}$} (B)
edge [loop above] node {${a\choose a}_{>0},{b\choose b}_{>0,+1},{\#\choose \#}_{>0,-1}$} (A)
edge node {${\#\choose \lozenge}_{>0,-1}$} (E);

\path (B)
edge [loop above] node {${\#\choose \#}_{>0,-1}$} (B)
edge node {${a \choose a}_{=0}$} (C);

\path (C)
edge [loop above] node {${a \choose a}$} (C)
edge node {${a \choose \lozenge}$} (D);

\path (E)
edge [loop left] node {${\#\choose \lozenge}_{>0,-1}$} (E)
edge node {${a \choose \lozenge}_{=0}$} (F);
\end{tikzpicture}
}
\caption{A deterministic non-blind $1$-counter automaton used in in Case 2.1 of the proof of Proposition \ref{prop:xone} when $R=M$. The start state is $q_0$ and accept states are $q_4$ and $q_6$.  The top path is followed when $e_m>1$ and the bottom when $e_m=1$.}
\label{fig:R=M}
\end{center}
\end{figure}
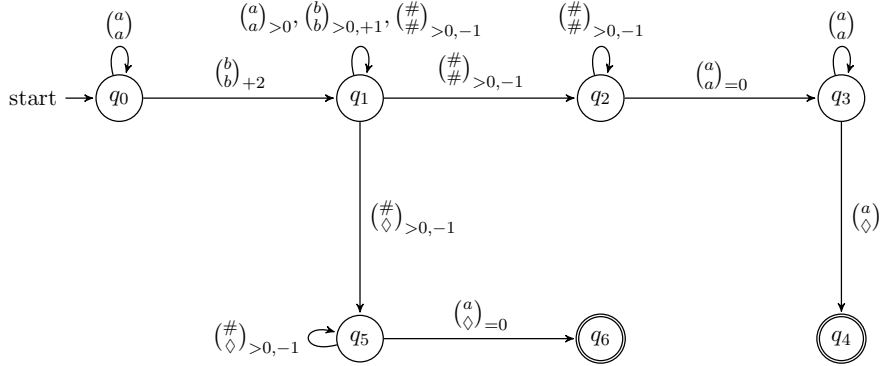

If $R<M$ then we must have $M=i_m$ and we consider three scenarios for the strings $u,v \in \Linf$.
\begin{enumerate}
\item The generator $x_R$ does not appear in the infinite normal form for $g$, and hence we have
    \begin{itemize}
    \item $u=\gamma \# a^{r_{R-1}}b^{s_{R-1}} \# \# \eta$, and
    \item $v=\gamma \# a^{r_{R-1}}b^{s_{R-1}} \# b\# \eta$.
    \end{itemize}
\item The generator $x_R$ does appear in the infinite normal form for $g$, as does $x_{R+1}$ and hence we have
\begin{itemize}
    \item $u=\gamma \# a^{r_{R}} \# a^{r_{R+1}} \# \eta$, and
    \item $v=\gamma \# a^{r_{R}} b\# a^{r_{R+1}} \# \eta$.
    \end{itemize}
\item The generator $x_R$ does appear in the infinite normal form for $g$ but  $x_{R+1}$ does not, and hence we have
\begin{itemize}
    \item $u=\gamma \# a^{r_{R}} \#  \# \eta$, and
    \item $v=\gamma \# a^{r_{R}-1} \# \eta$
    \end{itemize}
    as in Case 1.3.
\end{enumerate}
In all three cases above,  $\eta \subset \{a,\#\}^*$.

A finite state machine accepting convolutions of the above strings must check that the difference in the strings comes at the correct position; for this we require a single counter.

We now build a $1$-counter machine which accepts the language of all strings $\otimes(u,v)$ as above with $\gamma$ and $\eta$ replaced with an arbitrary string from $\{a,b,\#\}^*$.  Each $\otimes(u,v)$ in this language is the concatenation of a string from a prefix language with a string from a suffix language.  The prefix language is a non-blind deterministic $1$-counter language based on the same counter instructions as in Figure \ref{fig:R=M}, and below we explain why the counter must have value $0$ to transition between the two languages.  The suffix languages, one for each type above, are all regular according to Lemma \ref{L:RegularLanguages}.  The prefix and suffix languages are given in the following table, where $\gamma$ and $\eta$ denote words in $\{a,b,\#\}^*$.

\medskip
\begin{center}
\begin{tabular}{|l|l|l|l|}
\hline
Type of pair & Prefix Language & Suffix Language & FSA accepting \\
 & & & Suffix Language \\
\hline
(1) & $\{ \otimes(\gamma  a^nb^k \#\#,\gamma a^nb^k \#b)\}$  & $\{\otimes(\eta,\#\eta)\}$ & $M_1$  \\
 & with $n,k \geq 0$ & & \\
\hline
(2) & $\{ \otimes(\gamma \# a^n  \#,\gamma \# a^n b)\}$  & $S_2 = \{\otimes(a^k \#\eta,\#a^k\#\eta)\}$ &  $M_2$ \\
& with $n \geq 1$ &  with $k>0$ & \\
\hline
(3) & $\{ \otimes(\gamma \# a^n,\gamma \# a^n)\}$  & $\{\otimes(a \#\# \eta,\#\eta)\}$ & $M_3$ \\
 & with $n \geq 0$ & & \\
\hline
\end{tabular}
\end{center}

\medskip

Figure \ref{fig:R<M} presents a $1$-counter automaton which accepts the language of convolutions of concatenations of a prefix and suffix from the above table.  Any arrow terminating at a state labeled $M_i$ for $i=1,2,3$ is assumed to terminate at the start state of that machine.  The $\epsilon$ edges can be removed but are used to give the simplest depiction of the machine.

This automaton initially operates with the same counter instructions as in Figures \ref{fig:L4} and \ref{fig:R=M}, that is, ${b \choose b}_{+2}$ initially  followed by ${b \choose b}_{>0,+1}$ and ${\# \choose \#}_{>0,-1}$, as in Figures \ref{fig:L4} and \ref{fig:R=M}. Note that after reading a prefix in which there are $p$ copies of the symbol ${\#\choose\#}$, the value of the counter is equal to \[
1+f_0+f_{i_1}+\dots + f_{i_{p-1}}-p.\] In the case $R<M$, we have $1+f_0+f_1+\dots + f_{p-1}>p$ for all $p$, and we have $R = 1+f_0+f_1+\dots + f_n$.  If $\otimes(u,v)$ arises in the case $R<M$, the value of the counter will first equal $0$ when $R$ symbols of the form ${\# \choose \# }$ have been read.  If we divide $\otimes(u,v)$ into a prefix $p$ and suffix $s$ using the above prefix and suffix languages, the counter will first equal zero after the final ${\# \choose \# }$ in $p$ has been read.  Hence we verify in the machine in Figure \ref{fig:R<M} that the value of the counter has value $0$ after the final ${\# \choose \# }$ in the prefix word before transitioning to the suffix word.

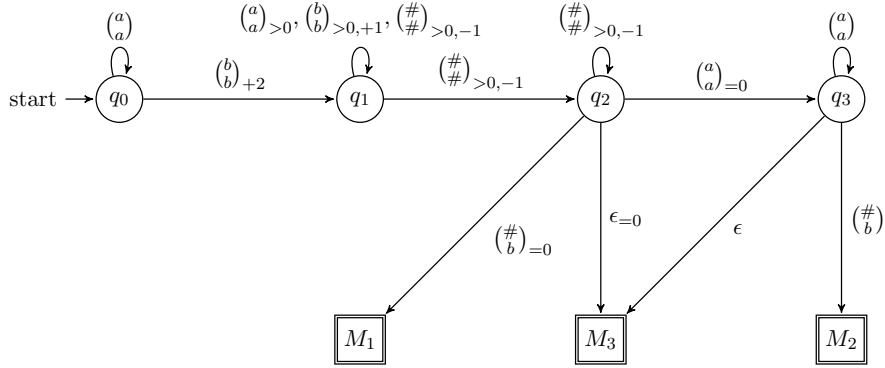
\begin{figure}[ht!]
\begin{center}
\scalebox{.8}{
\tikzset{every state/.style={minimum size=2em}} 
\begin{tikzpicture}[->,>=stealth',shorten >=1pt,auto, node distance=4cm,
semithick]

\node[state, initial] (S) {$q_0$};
\node[state] (A) [right of=S] {$q_1$};
\node[state] (B) [right of=A] {$q_2$};
\node[state] (C) [right of=B] {$q_3$};
\node[state, accepting,shape=rectangle] (E) [below of = A] {$M_1$};
\node[state, accepting,shape=rectangle] (G) [below of = B] {$M_3$};
\node[state, accepting,shape=rectangle] (D) [below of =C] {$M_2$};

\path (S)
edge [loop above] node {${a\choose a}$} (S)
edge node {${b \choose b}_{+2}$} (A);

\path (A)
edge node {${\#\choose \#}_{>0,-1}$} (B)
edge [loop above] node {${a\choose a}_{>0},{b\choose b}_{>0,+1},{\#\choose \#}_{>0,-1}$} (A);

\path (B)
edge [loop above] node {${\#\choose \#}_{>0,-1}$} (B)
edge node {${\# \choose b}_{=0}$} (E)
edge  node {$\epsilon_{=0}$} (G)
edge node {${a \choose a}_{=0}$} (C);

\path (C)
edge [loop above] node {${a \choose a}$} (C)
edge node {$\epsilon$} (G)
edge node {${\# \choose b}$} (D);
\end{tikzpicture}
}
\caption{A deterministic non-blind $1$-counter automaton which verifies that the change in $u$ and $v$ from the case $R<M$ occurs at the correct position in the word.  The start state is $q_0$ and accept states are the accept states from the finite state machines $M_1,M_2$ and $M_3$.  An arrow terminating at one of these machines is understood to terminate at the start state of the machine.}
\label{fig:R<M}
\end{center}
\end{figure}

Let $L_{R<M}$ denote the intersection of the language accepted by the machine in Figure \ref{fig:R<M} with $\otimes(\Linf,\Linf)$.  Then $L_1 = L_{R<M} \cup L_{R=M} \cup L_{R>M}$ is the language consisting exactly of those convolutions of strings described in Case 2.1.

\bigskip

\noindent
Case 2.2: For the remaining case, suppose that $1+f_0+f_1+\dots +f_{t-1} =t$ for some $t\leq j_n$.  Beginning with the expression in \eqref{eqn:infnf} for the infinite normal form of $g$, we can write the infinite normal form for $gx_1^{-1}$ as
\begin{equation}\label{eqn:x1-intermediatestep}
x_0^{e_0} \cdots x_{i_m}^{e_m}x_{i_n}^{-f_n} \cdots x_{j_{w+1}}^{-f_{w+1}} x_t^{-1} x_{j_w}^{-f_w} \cdots  x_0^{-f_0}.
\end{equation}
We again consider subcases, depending on whether or not the expression in \eqref{eqn:x1-intermediatestep} is the infinite normal form for $gx_1^{-1}$.  In each case, we show that the language $\{ \otimes(u,v)\}$ of accepted words in that case is the concatenation of a prefix language and a suffix language.  The suffix language is always a regular language.  The prefix language is always a deterministic non-blind $1$-counter language.  While it first seems like the language of all possible prefixes is also regular, we must use a counter to ensure that the difference between the strings $u$ and $v$ occurs at the proper place in the string.  We again use the counter instructions ${b \choose b}_{+2}$ followed by ${b \choose b}_{>0,+1}$ and ${\# \choose \#}_{>0,-1}$, as in Figures \ref{fig:L4} and \ref{fig:R=M}. Note that after reading a prefix in which there are $p$ copies of the symbol ${\#\choose\#}$, the value of the counter is equal to \[
1+f_0+f_{i_1}+\dots + f_{i_{p-1}}-p.\]
When  $1+f_0+f_{i_1}+\dots + f_{i_t}-t=0$, we are in the position to read the additional $x_t$ letter introduced by permuting the $x_1^{-1}$ past generators with smaller indices using the group relations.  The prefix string always terminates with the difference resulting from the newly introduced generator $x_t^{-1}$.  In Case 2.2.1, this difference is the removal of an $``a"$ symbol to obtain the reduced expression for $gx_16{-1}$.  In Case 2.2.2, the prefix string always terminates with the symbol ${ \# \choose b }$ where the $``b"$ corresponds to the $x_t^{-1}$.  Thus we check that the value of the counter is $0$ before transitioning to the suffix language.

\noindent
Case 2.2.1: If \eqref{eqn:x1-intermediatestep} is not the infinite normal form for $gx_1^{-1}$, then $t \neq j_{w+1}$, there is an index $p\leq m$ so that $i_p=t$, and $x_{t+1}$ is not present in the normal form for $g$ to any non-zero power, that is, $i_{p+1} \neq t+1$ and $j_{w+1} \neq t+1$. This case is analogous to Case 1.3 above, and we can write
\begin{itemize}
\item $u=\gamma a^{r_t} \# \# \eta$, and
\item $v= \gamma a^{r_t-1} \# \eta$.
\end{itemize}
Note that each $\otimes(u,v)$ of the above form can be written as a prefix $\otimes(\gamma a,\gamma \#)$ and a suffix from the language $\{ \otimes( \# \# \eta, \eta)\}$.  If we replace $\gamma$ and $\eta$ with any strings from $\{a,b,\#\}^*$, then it follows from Lemma \ref{L:RegularLanguages} that the language of all possible suffixes is regular. Let $M$ denote the finite state machine accepting these suffixes.  As explained above, we require that the value of the counter be $0$ to transition from a prefix word to a suffix word.

Figure \ref{fig:case2.2.1} contains a deterministic non-blind $1$-counter automaton accepting concatenations of prefix and suffix words of this form.  Let $L_{2.2.1}$ be the intersection of the language accepted by this machine with $\otimes(\Linf,\Linf)$.  Then $L_{2.2.1}$ is exactly the set of convolutions described in Case 2.2.1.

\begin{figure}[ht!]
\begin{center}
\scalebox{.8}{
\tikzset{every state/.style={minimum size=2em}} 
\begin{tikzpicture}[->,>=stealth',shorten >=1pt,auto, node distance=4cm,
semithick]

\node[state, initial] (S) {$q_0$};
\node[state] (A) [right of=S] {$q_1$};
\node[state, shape=rectangle] (B) [right of=B] {$M$};

\path (S)
edge [loop above] node {${a\choose a}$} (S)
edge node {${b \choose b}_{+2}$} (A);

\path (A)
edge node {${a\choose \#}_{=0}$} (B)
edge [loop above] node {${a\choose a}_{\geq 0},{b\choose b}_{>0,+1},{\#\choose \#}_{>0,-1}$} (A);

\end{tikzpicture}
}
\caption{A deterministic non-blind $1$-counter automaton which accepts exactly those strings from Case 2.2.1 with $\gamma,
\eta \in \{a,b,\#\}^*$.  The start state is $q_0$ and the accept states are the accept states of $M$.  Any arrow terminating at $M$ is assumed to terminate at the start state of $M$.}
\label{fig:case2.2.1}
\end{center}
\end{figure}
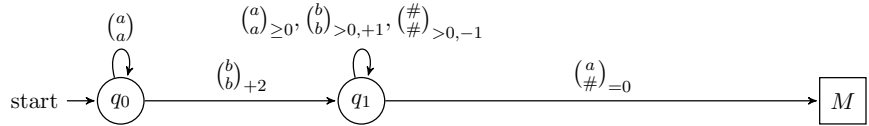

\noindent
Case 2.2.2: Now suppose that Equation \eqref{eqn:x1-intermediatestep} is the infinite normal form for $gx_1^{-1}$.  This occurs in three ways:
\begin{enumerate}
\item If $x_t^{-1}$ is already in the normal form of $g$, then
\begin{itemize}
\item $u=\gamma \# a^{r_t} b^{s_t} \# \eta$ with $s_t \neq 0$, and
\item $v= \gamma \# a^{r_t} b^{s_t+1} \# \eta$.
\end{itemize}

\medskip

\item If $x_t^{-1}$ is not in the infinite normal form for $g$, but $x_t$ and either $x_{t+1}$ or $x_{t+1}^{-1}$ are present, then
\begin{itemize}
\item $u=\gamma \# a^{r_t}  \# a^{r_{t+1}} b^{s_{t+1}} \# \eta$ with $r_t>0$ and $r_{t+1} +s_{t+1} > 0$, and
\item $v= \gamma \# a^{r_t}  b \# a^{r_{t+1}} b^{s_{t+1}} \# \eta$.
\end{itemize}

\medskip

\item If both $x_t$ and $x_t^{-1}$ are not present in the infinite normal form for $g$, then
\begin{itemize}
\item $u=\gamma \# a^{r_{t-1}} b^{s_{t-1}} \#  \# \eta$, and
\item $v= \gamma \# a^{r_{t-1}} b^{s_{t-1}} \# b \# \eta$.
\end{itemize}
\end{enumerate}

We now claim that the set of all strings of types $(1),(2)$ and $(3)$ above form a non-blind deterministic $1$-counter language. We again build a machine which accepts the language of such strings with $\gamma$ and $\eta$ replaced by any string in $\{a,b,\#\}^*$.  In each case we note that the set of such words can be divided into a prefix language which is a $1$-counter language and a suffix language which is regular, and the value of the counter must be $0$ to transition between the two languages.  The prefix and suffix languages for each type of pair $\otimes(u,v)$ listed in Case 2.2.2 are given in the following table.

\medskip

\begin{center}
\begin{tabular}{|l|l|l|l|}
\hline
Type of pair & Prefix Language & Suffix Language & FSA accepting \\
 & & & Suffix Language \\
\hline
(1) & $\{ \otimes(\gamma  \# a^nb^k \#,\gamma \# a^nb^{k+1})\}$  & $\{\otimes(\eta,\#\eta)\}$ & $M_1$  \\
 & with $n \geq 0, \ k \geq 1$ & & \\
\hline
(2) & $\{ \otimes(\gamma \# a^n  \#,\gamma \# a^n b)\}$  &  $S_2 = \{\otimes(a \eta,\# a \eta)\}$  &  $M_2$ \\
& with $n \geq 1$ & \hphantom{aaaa} $\cup \{\otimes(b \eta,\# b \eta)\}$ & \\
\hline
(3) & $\{ \otimes(\gamma \# \#,\gamma \# b)\}$  & $\{\otimes(\eta,\#\eta)\}$ & $M_1$  \\
  & & & \\
\hline
\end{tabular}
\end{center}

\medskip

The reasoning given at the beginning of Case 2.2 implies that if $\otimes(u,v)$ is a pair in this case, then the value of the counter at the end of the prefix string is equal to zero  This value is verified immediately after the $t$-th ${ \# \choose \#}$ symbol is read by the machine in the prefix word.

\begin{figure}[ht!]
\begin{center}
\scalebox{.8}{
\tikzset{every state/.style={minimum size=2em}} 
\begin{tikzpicture}[->,>=stealth',shorten >=1pt,auto, node distance=4cm,
semithick]

\node[state, initial] (S) {$q_0$};
\node[state] (A) [right of=S] {$q_1$};
\node[state] (B) [right of=A] {};
\node[state] (C) [below of=A] {};
\node[state, shape=rectangle] (D) [below of = C] {$M_1$ (Type (1))};
\node[state, shape=rectangle] (E) [below of = S] {$M_1$ (Type (3))};
\node[state, shape=rectangle] (F) [below of = B] {$M_2$ (Type (2))};

\path (S)
edge [loop above] node {${a\choose a}$} (S)
edge node {${b\choose b}_{+2}$} (A);

\path (A)
edge [loop  above] node {${a\choose a}_{>0},{b\choose b}_{>0,+1},{\#\choose \#}_{>0,-1}$} (A)
edge node {${a\choose a}_{=0}$} (B)
edge node {${b \choose b}_{=0}$} (C)
edge node {${\# \choose b}_{=0}$} (E);

\path (B)
edge  node {${b\choose b}$} (C)
edge [loop] node {${a\choose a}$} (B)
edge node {${\# \choose b}$} (F);

\path (C)
edge  node {${\#\choose b}$} (D)
edge [loop left] node {${b \choose b}$} (C);

\end{tikzpicture}}
\caption{A deterministic non-blind $1$-counter automaton used in Case 2.2.2 of the proof of Proposition \ref{prop:xone}.  The start state is $q_0$; any arrow terminating at $M_1$ or $M_2$ is assumed to end at the start state of the appropriate machine.}
\label{fig:L8}
\end{center}
\end{figure}
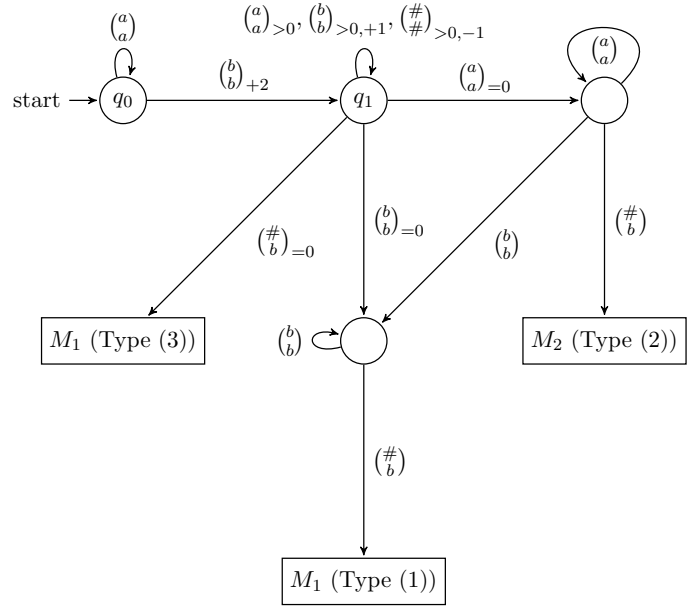

Figure \ref{fig:L8} contains a deterministic non-blind $1$-counter language accepting concatenations of prefix and suffix words in this case, with arbitrary strings $\gamma$ and $\eta$.  Let $L_{2.2.2}$ be the intersection of the language accepted by the machine in Figure \ref{fig:L8} with the regular language $\otimes(\Linf,\Linf)$.  Then the language of all $\otimes(u,v)$ accepted in Case 2.2 is exactly $L_2 = L_{2.2.1} \cup L_{2.2.2}$ which is a deterministic non-blind $1$-counter language.

Let $L_{s_0 \neq 0}=L_1 \cup L_2$, and it follows that $L_{x_1^{-1}} = L_{s_0 = 0} \cup L_{s_0 \neq 0}$ and is a deterministic non-blind $1$-counter language, as required.
\end{proof}

We have now proven the following theorem.
\begin{theorem}\label{thm:inf-normal-form-lang}
Thompson's group $F$ is deterministic non-blind $1$-counter graph automatic with respect to the generating set
$X = \{x_0^{\pm 1},x_1^{\pm 1} \}$ and symbol alphabet $\{a,b,\#\}$.
\end{theorem}

Since 1-counter languages are (strictly)
contained in the class of context-free languages, we obtain the corollary that $F$ is context-free graph automatic.

\begin{lemma}\label{lem:notRegLinfty}
$F$ is not graph automatic with respect to the normal form $\Linf$ and alphabet $\{a,b,\#\}$ given above.
\end{lemma}
\begin{proof}
Suppose $L_{x_1^{-1}}$ was a regular language, with  pumping length $p$.
The string $\otimes(b^p, b^p\#^{p+1}b)$ is in the language, representing
\begin{itemize}
\item $\bar{u}=_F x_0^{-p}$, and
\item $\bar{v}=_Fx_0^{-p}x_1^{-1}=_F x_{p+1}^{-1}x_0^{-p}$.
\end{itemize}
By the Pumping Lemma for Regular Languages, $\otimes(b^p, b^p\#^{p+1}b)$  can be partitioned into $xyz$ with $|xy|\leq p, |y|>0$ and $xy^iz\in L_{x_1^{-1}}$ for all $i\in\mathbb N$.  However, the convolution $\otimes(b^{p+m}, b^{p+m}\#^{p+1}b)$, for any $m>0$, does not lie in $L_{x_1^{-1}}$ as the second word does not have the correct number of $\#$ symbols, that is, $x_0^{-(p+m)}x_1^{-1} \neq x_{1+p}^{-1} x_0^{p+m}$. Hence $L_{x_1^{-1}}$ is not a regular language.
\end{proof}

\bibliographystyle{plain}
\bibliography{refs}

\end{document}